\documentclass[11pt]{article}
\usepackage{latexsym,subfigure}
\usepackage{comment}
\usepackage{amsmath,amsthm,amsfonts,amssymb,graphicx,epsfig,latexsym,color}
\usepackage{epsf,enumerate}

\newtheorem{thm}{Theorem}[section]
\newtheorem{lemma}[thm]{Lemma}
\newtheorem{cor}[thm]{Corollary}
\newtheorem{prop}[thm]{Proposition}
\newtheorem*{main}{Main Theorem}{}
\newtheorem*{question}{Question}{}

\theoremstyle{remark}

\newtheorem{remark}[thm]{Remark}

\newtheorem*{definition*}{Definition}
\newtheorem*{remark*}{Remark}

\def\C{{\mathcal C}}
\def\X{{\mathcal X}}
\def\R{{\mathbb R}}

\def\Z{{\mathbb Z}}

\def\G{{\Gamma}}

\def\MCG{{\it MCG\ }}

\begin{document}

\title{A hyperbolic $Out(F_n)$-complex}
  \author{Mladen Bestvina and Mark Feighn \thanks{Both authors
  gratefully acknowledge the support by the National Science
  Foundation. }}
  \date{October 28, 2009}
\maketitle

\abstract{For any finite collection $f_i$ of fully irreducible automorphisms
  of the free group $F_n$ we construct a connected $\delta$-hyperbolic
  $Out(F_n)$-complex in which each $f_i$ has positive translation length.}

\section{Introduction}

The study of the outer automorphism group $Out(F_n)$ of a free group
$F_n$ of rank $n$ has very successfully been driven by analogies
with mapping class groups. At the foundation of the theory is
Culler-Vogtmann's Outer space \cite{CV}, which plays the role of
Teichm\"uller space. The topology of Outer space is very well
understood, but its geometry is still very much a mystery. This is to
be contrasted with the rich theory of the geometry of Teichm\"uller
space. An instance of this contrast is the celebrated result of Masur
and Minsky \cite{MM} that the curve complex is hyperbolic. There is no
analogous result in the $Out(F_n)$ category, although candidates for
such a complex abound, see \cite{ilya-lustig}.

In this paper we prove the following, where $\overline{\mathcal{PT}}$
denotes the compactified Outer space.

\begin{main}
  For any finite collection $f_1,\cdots,f_k$ of fully irreducible
  elements of $Out(F_n)$ there is a connected $\delta$-hyperbolic
  graph $\X$ equipped with an (isometric) action of $Out(F_n)$ such
  that
\begin{itemize}
\item the stabilizer in $Out(F_n)$ of a simplicial tree in
  $\overline{\mathcal{PT}}$ has bounded
  orbits, 
\item the stabilizer in $Out(F_n)$ of a proper free factor $F\subset F_n$
  has bounded orbits, and
\item
  $f_1,\cdots,f_k$ have nonzero translation lengths.
\end{itemize}
\end{main}

The situation is much less than ideal, not
only because of the dependence of $\X$ on choices, but also because there is
no ``intrinsic'' description of the complexes in the style of the
curve complex.

However, the complexes are useful in that they allow construction of
many 
quasi-homomorphisms on $Out(F_n)$, a result recently announced by Ursula
Hamenst\"adt. 

The construction follows an idea of Brian Bowditch, who used it to
show that convergence groups are hyperbolic \cite{bo}. In Section 2 we
review Bowditch's construction, in Section 3 we sketch the analogous
construction of a hyperbolic complex for mapping class groups and in
Section 4 we carry out this program for $Out(F_n)$. The construction
for $Out(F_n)$ relies on the dynamics of the action of $Out(F_n)$ on
spaces of trees and currents, and we review the necessary
material. Some of the results we need are slight variations of the
ones found in the literature, and we sketch the proofs of these. 

{\bf Acknowledgements:} The first author would like to thank Ursula
Hamenst\"adt for a very intriguing lecture in Baltimore in May 2008
that inspired this work. We also thank Ken Bromberg for his help with
Remark \ref{ken}.

\section{Bowditch's construction}
The goal of this section is to show that if a group $\G$ acts on a
space $M$ satisfying some simple axioms, then $\G$ also acts on a
$\delta$-hyperbolic space. The model situation is that of a
convergence group action on a compact space, discussed by Bowditch
\cite{bo}. He proved that a group $\G$ that acts on a compact
metrizable space as a convergence group (i.e. properly and cocompactly
on the space of triples of distinct points) is hyperbolic and the
compact space is equivariantly homeomorphic to the boundary $\partial
\G$. In fact, with very little modification, Bowditch's construction
applies to noncompact spaces. For example, by looking at the action of
the mapping class group on the (suitable subset of the) Thurston
boundary, this gives its action on a hyperbolic graph (it is not clear
how this graph is related to the curve complex).

We will outline Bowditch's construction. First, we recall some
definitions. Fix an action (by homeomorphisms) of a group $\G$ on a
(metrizable) space $M$. We will assume that $M$ has no isolated
points. 

\subsection{Annulus systems}

An {\it annulus} in $M$ is a pair $A=(A^-,A^+)$ of disjoint closed
subsets of $M$ whose union is not all of $M$. For a subset $K\subset
M$ write $K<A$ if $K\subset int\ A^-$ and write $A<K$ if
$K<-A:=(A^+,A^-)$. For annuli $A$ and $B=(B^-,B^+)$ write $A<B$ if
$int\ A^+\cup int\ B^-=M$.

An {\it annulus system} $\mathcal A$ is a $\G$-invariant set of
annuli, such that $A\in\mathcal A$ implies $-A\in\mathcal A$.

If $K,L\subset M$ write $(K|L)=n\in \{0,1,\cdots,\infty\}$ if $n$ is
the maximal number of annuli $A_i$ in $\mathcal A$ such that 
$$K<A_1<A_2<\cdots<A_n<L$$
For finite sets we drop braces, e.g. $(ab|cd)$ means
$(\{a,b\}|\{c,d\})$. 

Consider the following axioms.

\begin{enumerate}
\item [(A1)] If $x\neq y$ and $z\neq w$ then $(xy|zw)<\infty$.
\item [(A2)] There is $k\geq 0$ such that for any $x,y,z,w\in M$ either
  $(xz|yw)\leq k$ or $(xw|yz)\leq k$.
\end{enumerate}

\subsection{Hyperbolic crossratio}\label{s:crossratio}

A {\it crossratio} on $M$ is a function $M^4\to [0,\infty]$,
$(x,y,z,w)\mapsto (xy|zw)$ such that $(xy|zw)=(yx|zw)=(zw|xy)$. A
crossratio is {\it $k$-hyperbolic} if: 
\begin{enumerate}
\item [(C1)] If $F\subset M$ is a 4-element subset, we can write
  $F=\{x,y,z,w\}$ such that $(xz|yw)\simeq_k0$ and $(xw|yz)\simeq_k0$.
\item [(C2)] If $F\subset M$ is a 5-element subset, we can write
  $F=\{x,y,z,w,u\}$ such that $(xy|zu)\simeq_k(xy|wu)$,
  $(xu|zw)\simeq_k(yu|zw)$, $(xy|zw)\simeq_k(xy|zu)+(xu|zw)$.
\end{enumerate}
where $a\simeq_kb$ means $|a-b|\leq k$. The intuition is that if $M$
is a tree, then letting $(xy|zw)$ be the distance between $[x,y]$ and
$[z,w]$ defines a 0-hyperbolic crossratio.

Note that $(C1)$ implies that for any 4 distinct points at most one of
the crossratios $(xy|zw),(xz|yw),(xw|yz)$ is $>k$. We write $(xy:zw)$
to mean that $(xz|yw)$ and $(xw|yz)$ are $\leq k$. We also write
$(xy:u:zw)$ to mean $(xy:zw),(yu:zw),(xu:zw),(xy:uw),(xy:uz)$.

A hyperbolic crossratio is a {\it path crossratio} if for any distinct
$x,y,z,w\in M$ and any $p\leq (xy|zw)$ there is $u\in M$ with
$(xy:u:zw)$ and $(xy|zu)\simeq p$, where we write $\simeq$ for
$\simeq_k$ when $k$ is understood.

\begin{prop}
Suppose the annulus system $\mathcal A$ satisfies (A1) and (A2). Then
the crossratio $(xy|zw)$ defined by counting annuli is a hyperbolic
path crossratio.
\end{prop}

\begin{proof}
This is \cite[Proposition 6.5]{bo}. Note that Bowditch assumes that
$M$ is compact, but in fact he does not use this assumption in the
proof. 
\end{proof}

\subsection{Hyperbolic path quasi-metric}

Let $Q$ be the set of ordered triples of distinct points in $M$. We
assume that we are given hyperbolic path crossratio on $M$.

If $A=(a_1,a_2,a_3)\in Q$ and $B=(b_1,b_2,b_3)\in Q$ define

$$\rho(A,B)=\max (a_ia_j|b_kb_l)$$
over $i\neq j$ and $k\neq l$. 

The intuition is that one can embed the 6 points $a_i,b_j$ into a
metric tree so that the crossratios get distorted a bounded amount. Then
$\rho(A,B)$ (up to a bounded number) is the distance between the centers
of the tripods spanned by $a_i$ and by $b_j$ respectively.

\begin{prop}
$(Q,\rho)$ is a hyperbolic path quasi-metric space. This means that
  for some $k\geq 0$
\begin{itemize}
\item (quasi-metric) $\rho(A,C)\leq \rho(A,B)+\rho(B,C)+k$,
\item (hyperbolic) the 4-point definition of $k$-hyperbolicity holds (via
  Gromov products),
\item (path) Any two points $A, B$ can be connected by a finite sequence
  $A=z_0, z_1, \dots, z_N=B$ so that $\rho(z_i,z_{j})\simeq_k |i-j|$.
\end{itemize}
\end{prop}

\begin{proof} Proposition 4.2 and Lemma 4.3 of \cite{bo}.
\end{proof}

\begin{prop}
  If $(Q,\rho)$ is a hyperbolic path quasi-metric, for large $r>0$
  define the graph $G_r(Q)$ whose vertices are the points of $Q$ and
  two vertices $A,B$ are connected by an edge if $\rho(A,B)\leq
  r$. Then $G_r(Q)$ is a connected $\delta$-hyperbolic graph
  quasi-isometric to $(Q,\rho)$.
\end{prop}

\begin{proof}
See the discussion before Lemma 3.1 in \cite{bo}.
\end{proof}

We shall refer to the graph $\X=G_r(Q)$ as the {\it Bowditch
  complex}. Note that the set of vertices of $\X$ is equipped with the
edge-path metric $d$ as well as with the quasi-metric $\rho$.

\section{A hyperbolic complex for \MCG}

Let \MCG denote the mapping class group of a fixed compact connected
surface. The standard reference for the material in this section is
\cite{FLP}. 
To motivate some of the arguments in the Main Theorem, we start by
discussing the
(somewhat simpler) version for \MCG.

\begin{thm}\label{MCG}
For any finite collection $f_1,\cdots,f_k$ of pseudo-Anosov mapping
classes there is connected $\delta$-hyperbolic graph $\X$ and an action
of \MCG such that:
\begin{itemize}
\item the stabilizer of a simple closed curve has bounded orbits, and
\item $f_1,\cdots,f_k$ have nonzero translation lengths.
\end{itemize}
\end{thm}

We will call $\X$ the {\it Bowditch complex for }\MCG.

\subsection{Verifying (A1)-(A2)}
For now the space $M$ is Thurston's boundary $\mathcal{PML}$, i.e. the
space of projective measured laminations; it will
be made smaller later. Let $\Lambda^\pm_{i}$ be the stable and
unstable laminations for $f_i$ and choose small
neighborhoods $D_i^\pm$ of $\Lambda^\pm_i$
forming an annulus $A_i$. The annulus system consists of the
translates of $\pm A_i$ for $i=1,2,\cdots,k$.
By $I(-,-)$ denote the intersection number and by $L(\cdot)$
the length with respect to a fixed hyperbolic structure.

In this section we verify (A1) and (A2) (after passing to a smaller
$M$). To simplify notation, we assume $k=1$ and drop subscripts.

\begin{lemma}\label{1}
If $a$ and $b$ are simple closed curves then $(a|b)<\infty$.
\end{lemma}

\begin{proof}
Suppose $g\in MCG$ such that $a<g(A)<b$. Then $g^{-1}(a)\in D^-$ and
$g^{-1}(b)\in D^+$. 
The expression
$$\frac{I(g^{-1}(a),g^{-1}(b))}{L(g^{-1}(a))L(g^{-1}(b))}$$
does not change if we scale $g^{-1}(a)$ or $g^{-1}(b)$. It follows, by
the continuity of $I$ and $L$, that this expression is close to 
$$\mu=\frac{I(\Lambda^+,\Lambda^-)}{L(\Lambda^+)L(\Lambda^-)}>0$$ and
in particular it is bounded away from 0.
But $I(g^{-1}(a),g^{-1}(b))=I(a,b)$ is fixed, so
it follows that both $L(g^{-1}(a))$ and $L(g^{-1}(b))$ are uniformly
bounded. Since $a$ and $b$ fill, there are only finitely many such
$g$.
\end{proof}

\begin{remark}\label{universal bound}
Note that when $a,b$ are disjoint simple closed curves then
$(a|b)=0$. This is because $u\in D^-$, $v\in D^+$ implies $I(u,v)>0$.
\end{remark}

\begin{lemma}\label{2}
Let $a_1,a_2,b_1,b_2$ be laminations in $\mathcal{PML}$ with
$I(a_1,a_2)>0$ and $I(b_1,b_2)>0$. Then $(a_1a_2|b_1b_2)<\infty$.
\end{lemma}

\begin{proof}
Consider some $g\in MCG$ with $a_i<g(A)<b_j$, $i,j=1,2$. Then
$g^{-1}(a_i)\in D^-$ and $g^{-1}(b_j)\in D^+$. As above we have
$$I(g^{-1}(a_i),g^{-1}(b_j))\geq K~ L(g^{-1}(a_i))L(g^{-1}(b_j))$$
If there are infinitely many such $g$ then one of
the following cases occurs:

{\bf Case 1.} $L(g^{-1}(a_i))$ and $L(g^{-1}(b_j))$ are bounded above for
some choice $i,j\in\{1,2\}$ (and a subsequence of the $g$'s). Choose a
curve $c$ and note that both intersection numbers $I(c,g^{-1}(a_i))$
and $I(c,g^{-1}(b_j))$ are bounded, i.e. $I(g(c),a_i)$ and
$I(g(c),b_j)$ are both bounded (note that $I(\Lambda,\Lambda')\leq C~
L(\Lambda)L(\Lambda')$ for any two laminations, where $C$ is a
constant that depends only on the underlying hyperbolic
surface). Since $a_i$ and $b_j$ fill, this implies that $L(g(c))$ is
bounded. Since this is true for any $c$ it follows that there are only
finitely many $g$'s, contradiction.

{\bf Case 2.} Either $L(g^{-1}(a_i))\to 0$ for both $i=1,2$ or
$L(g^{-1}(b_j))\to 0$ for both $j=1,2$ (over a subsequence of
$g$'s). Say the former. Then $I(g^{-1}(a_1),g^{-1}(a_2))\to 0$,
i.e. $I(a_1,a_2)=0$, contradiction.
\end{proof}

There is also a hybrid situation:

\begin{lemma}\label{hybrid}
If $a$ is a curve and $b_1,b_2$ are laminations with $I(b_1,b_2)>0$
then $(a|b_1b_2)<\infty$.
\end{lemma}

\begin{proof}
Similar to the other two lemmas. We have
$$I(g^{-1}(a),g^{-1}(b_i))\geq K~ L(g^{-1}(a))L(g^{-1}(b_i))$$ for
$j=1,2$. 
There are now two cases.

{\bf Case 1.} $L(g^{-1}(a))$ stays bounded. Then both $L(g^{-1}(b_i))$,
$i=1,2$ are bounded as well by the above inequality 
Since $a$ and
$b_1$ fill, this restricts $g$ to a finite set, as in Case 1 of Lemma
\ref{2}.

{\bf Case 2.} $L(g^{-1}(a))\to\infty$. Then $L(g^{-1}(b_i))\to 0$,
$i=1,2$ and hence $I(b_1,b_2)=0$ as in Case 2 of Lemma \ref{2}.
\end{proof}

\begin{lemma}\label{A2}
If $D^\pm$ are chosen to be small enough then for any $a,b,c,d\in
\mathcal{PML}$ we have $(ac|bd)=0$ or $(ad|bc)=0$.
\end{lemma}

\begin{proof}
Scale each lamination in $\mathcal{PML}$ so that its length is 1 (with respect
to a fixed hyperbolic metric). Choose $D^\pm$ so that when
$x,y\in D^+$ (or $D^-$) then $I(x,y)<\epsilon$ and if $x\in D^+$ and
$y\in D^-$ then $|I(x,y)-I(\Lambda^+,\Lambda^-)|<\epsilon$. This is
possible by the continuity of the intersection number. 
Note that we could also write e.g.
$$\frac{I(x,y)}{{L(x)L(y)}}<\epsilon$$
for the first inequality and this is invariant under scaling $x$ and
$y$. 

Now assume $(ac|bd)>0$ and $(ad|bc)>0$. Then for some $g_1,g_2\in
MCG$ we have $a_1,c_1,a_2,d_2\in D^-$, $b_1,d_1,b_2,c_2\in D^+$, where
$a_i=g_i(a)$ etc. Thus we have
$$\frac{I(a_1,c_1)}{{L(a_1)L(c_1)}}<\epsilon$$
and
$$\frac{I(a_2,c_2)}{{L(a_2)L(c_2)}}\sim \mu$$
where
$\mu=\frac{I(\Lambda^+,\Lambda^-)}{L(\Lambda^+)L(\Lambda^-)}$. Dividing
the two inequalities and noting that $I(a_1,c_1)=I(a_2,c_2)$ gives
$$\frac{L(a_2)L(c_2)}{L(a_1)L(c_1)}\overset<\sim\epsilon/\mu$$
Similarly we have
$$\frac{L(b_2)L(d_2)}{L(b_1)L(d_1)}\overset<\sim\epsilon/\mu\quad\quad
\frac{L(a_1)L(d_1)}{L(a_2)L(d_2)}\overset<\sim\epsilon/\mu\quad\quad
\frac{L(b_1)L(c_1)}{L(b_2)L(c_2)}\overset<\sim\epsilon/\mu$$
Multiplying gives the contradiction $1\overset<\sim\epsilon^4/\mu^4$
(for small $\epsilon$).
\end{proof}

\begin{proof}[Proof of Theorem \ref{MCG}]
Let $M\subset\mathcal{PML}$ be the subset consisting of stable laminations
$\Lambda_g^+$ as $g$ varies over all pseudo-Anosov
homeomorphisms in \MCG. The annulus system will be the restriction to
$M$ of the annulus system considered above. Since distinct elements of
$M$ have nonzero intersection number, Lemma \ref{2} verifies (A1), and
Lemma \ref{A2} verifies (A2).
The resulting Bowditch complex $\X$ is hyperbolic. The statements
about orbits and translation lengths are verified in the next section.
\end{proof}

\subsection{Orbits in $\X$}

\begin{prop}\label{orbits}
  The stabilizer in \MCG of a simple closed curve has bounded
  orbits. The original pseudo-Anosov homeomorphisms $f_i$ have nonzero
  translation lengths.
\end{prop}

\begin{proof}
By construction,
$(\Lambda_f^+|\Lambda_f^-)>0$. Then by pumping (i.e. using North-South
dynamics) 
$(\Lambda_f^+|\Lambda_f^-)=\infty$ and in fact
$$d((a,b,c),(f^m(a),f^m(b),f^m(c)))$$ grows linearly. This proves that
$f$ has nonzero translation length.

Now consider the stabilizer $S_a$ of a curve $a$. Fix a triple
$(p_1,p_2,p_3)\in \X$. By Lemma \ref{hybrid} we know that
$N=\max_{i\neq j}(a|p_ip_j)<\infty$. If $g\in S_a$, consider a
collection of $D$ disjoint annuli separating $p_i,p_j$ from
$g(p_u),g(p_v)$ for some $i\neq j$ and $u\neq v$. At most one of these
contains $a$ ($B=(B^-,B^+)$ {\it contains} $a$ if $a\in M-(B^-\cup
B^+)$); remove it. Thus at least $\frac{D-1}2$ separate $a$ from
$p_i,p_j$ or from $g(p_u),g(p_v)$ and we deduce $N\geq
(a|p_ip_j)\geq\frac{D-1}2$ or $(a|g(p_u)g(p_v))\geq\frac{D-1}2$. But
$(a|g(p_u)g(p_v))=(a|p_up_v)\leq N$ since $g(a)=a$, so in any case
$D\leq 2N+1$.
\end{proof}

\begin{remark}\label{orbitsize}
This argument shows that the orbit of $A=(p_1,p_2,p_3)$ under $S_a$
has ($\rho$-)diameter at most $2\max_{i\neq j}(a|p_ip_j)+1$. Note also that
$\max_{i\neq j}(a|p_ip_j)$ is a lower bound for the diameter of the
orbit, by considering iterations by the Dehn twist in $a$ and using
the fact that the iterates of $p_1,p_2,p_3$ converge to $a$.
\end{remark}

\begin{remark}
In the above proof we used the triangle type inequality
$$(A|B)\leq (A|x)+(x|B)+1$$
where $x\in M$ and $A,B\subset M$. This is \cite[Lemma 6.1]{bo}.
\end{remark}

\subsection{Comparing the Bowditch complex $\X$ and the curve complex
  $\C$} \label{comparison}

Recall that if a group acts isometrically on a $\delta$-hyperbolic
geodesic space with bounded orbits, then there is an orbit
of diameter $\leq 8\delta$. Thus there are some $K_\rho>0$ and $K_d>0$
such that the stabilizer $S_a\subset MCG$ of $a$ has an orbit of vertices
in $\X$ of $\rho$-diameter $\leq K_\rho$ and $d$-diameter $\leq K_d$,
for any curve $a$ (e.g. $K_d$ can be taken to be $8\delta+1$ if $\X$ is
$\delta$-hyperbolic with respect to $d$).

Define $\Phi:\C\to \X$ by the rule that $\Phi(a)$ is a triple $(p,q,r)$
that belongs to such an orbit.

\begin{lemma}\label{Phi}
$\Phi$ is coarsely well defined and it is Lipschitz.
\end{lemma}

\begin{proof}
We need to check that different choices for $\Phi(a)$ are close, but
this is a special case of the Lipschitz condition. Suppose $a,b$ are
disjoint curves and let $\Phi(a)=(p_1,p_2,p_3)$,
$\Phi(b)=(q_1,q_2,q_3)$. Then $\max_{i\neq j}(a|p_ip_j)\leq K_\rho$
and $\max_{u\neq v}(b|q_uq_v)\leq K_\rho$ and since $(a|b)=0$ 
(see Remark \ref{universal bound})
we have $$\max_{i\neq j}(p_ip_j|b)\leq K_\rho+1$$ by the triangle
inequality. Thus 
\begin{multline*}
\rho(p_1p_2p_3,q_1q_2q_3)=\max_{i\neq j,u\neq
  v}(p_ip_j|q_uq_v)\leq\\ \max_{i\neq j}(p_ip_j|b)+\max_{u\neq
  v}(b|q_uq_v)+1\leq 2K_\rho+2
\end{multline*}
\end{proof}

\begin{lemma}\label{coarsely onto}
If $D^{\pm}$ are chosen small enough then $\Phi$ is coarsely onto.
\end{lemma}

\begin{proof}
  Let $(p_1,p_2,p_3)$ be a triple in $M$. Scale them so that all 3
  intersection numbers are equal, say to $1$. Then by Bowditch's lemma
  (see Lemmas 5.1 and 5.3 of \cite{bo:hyp}, or for another exposition
  see \cite[Lemma 5.7]{glasnik}) there is a curve $a$ such that
  $I(a,p_i)\leq R$, $i=1,2$ for a constant $R$ that depends only on
  the surface, and moreover, $I(a,z)\leq R(I(p_1,z)+I(p_2,z))$ for all
  laminations $z$ (this result is stated in the above references only
  for multicurves, but it extends easily to laminations). By putting
  $z=p_3$ we see that $I(a,p_3)\leq 2R$. (By the Bowditch's proof of
  hyperbolicity of $\C$, $a$ is near
  the center of the ideal triangle in the curve complex with vertices
  at infinity corresponding to $p_1,p_2,p_3$.)

We now claim that $(a|p_ip_j)=0$ for $i\neq j$ if $D^{\pm}$ are chosen
small. Thus $\Phi(a)$ is close to $(p_1,p_2,p_3)$.

Suppose $(a|p_ip_j)>0$. Then there is $g\in MCG$ so that $\tilde
a=g^{-1}(a)\in D^-$ and $\tilde p_i=g^{-1}(p_i)\in D^+$, $\tilde
p_j=g^{-1}(p_j)\in D^+$. As in Lemma \ref{A2} we have:
\begin{equation}
\frac{I(\tilde a,\tilde p_i)}{{L(\tilde a)L(\tilde
    p_i)}}\sim\mu\qquad\frac{I(\tilde a,\tilde p_j)}{{L(\tilde a)L(\tilde
    p_j)}}\sim\mu 
\end{equation}
and 
\begin{equation}
\frac {I(\tilde p_i,\tilde p_j)}{{L(\tilde p_i)L(\tilde
    p_j)}}<\epsilon
\end{equation}
Dividing (2) by each of the equations in (1) and taking into account
$I(\tilde p_i,\tilde p_j)=I(p_i,p_j)=1$ and $I(\tilde a,\tilde
p_i)=I(a,p_i)\leq 2R$, $I(\tilde a,\tilde
p_j)\leq 2R$ gives
$$L(\tilde a)\overset\sim <2R\sqrt\epsilon/\mu$$
which is a contradiction for small $\epsilon$ (we are on a fixed
hyperbolic surface).
\end{proof}

\begin{remark}\label{ken}
  Suppose $f$ and $g$ are two pseudo-Anosov elements of \MCG such that
  no nontrivial power of $f$ is conjugate to a power of $g$, then
  there are neighborhoods $U^\pm$ of $\Lambda_f^\pm$ such that there
  is no $h\in MCG$ with $h(\Lambda_g^\pm)\in U^\pm$. There are two
  proofs of this claim, both modelled on the proof of a similar
  assertion for two hyperbolic elements in a discrete subgroup of
  $SO(n,1)$. Indeed, the existence of $h$ forces the geodesics
  associated to $f$ and $g$ in the orbit space to be very close, which
  is impossible since they are distinct closed geodesics (as sets).
  There are two variants of this argument for \MCG, one using the
  Teichm\"uller metric and the other the Weil-Petersson metric. If
  $\mathcal F_f^\pm$ are the stable and unstable measured foliations
  associated with $f$, then the Teichm\"uller axis $A_f$ of $f$
  consists of conformal structures obtained from Euclidean metrics
  with singularities of the form $ds^2=e^{2t}d\mu_+^2+e^{-2t}d\mu_-^2$
  where $\mu_\pm$ are the measures on $\mathcal F_f^\pm$ and
  $t\in\R$. If $\mathcal F_{h_igh_i^{-1}}^\pm \to \mathcal F_f^\pm$ as
  measured foliations then clearly $A_{h_igh_i^{-1}}\to A_f$ uniformly on
  compact sets and we have a contradiction as before. The
  Weil-Petersson version uses the fact that pseudo-Anosov elements of
  \MCG have unique geodesic axes plus a theorem of Brock-Masur-Minsky
  \cite[Corollary 1.6]{BMM}, which says that axes are close when the
  associated stable and unstable laminations are close in
  $\mathcal{PML}$.

It follows that whenever we are given $g_1,g_2,\cdots,g_m\in \MCG$ such
that nontrivial powers of $g_j$ are not conjugate to powers of $f_i$'s
then by choosing $D_i^\pm$'s sufficiently small we can arrange that
the $g_j$'s have bounded orbits. This is because
$(\Lambda_{g_j}^+|\Lambda_{g_j}^-)=0$ by construction, so if we take
any $\Lambda\in\mathcal{PML}-\{\Lambda_{g_j}^\pm\}$ and set
$x=(\Lambda_{g_j}^+,\Lambda_{g_j}^-,\Lambda)$ then
$\rho(x,g_j^N(x))=0$ for any $N$. 

Now suppose that $\mathcal X_1,\mathcal X_2,\cdots$ is a sequence of
Bowditch complexes obtained by taking smaller and smaller
neighborhoods $D_i^\pm$. All $\mathcal X_i$ have the same vertex sets,
but for $i<j$ $\mathcal X_j$ in general has more edges than $\mathcal
X_i$, so we have natural maps $\mathcal X_1\to \mathcal
X_2\to\cdots$. One may wonder whether eventually this sequence
consists of quasi-isometries (all maps are clearly coarsely onto). The
answer to this question is negative. To see this, choose some
pseudo-Anosov homeomorphism $g$ whose stable and unstable laminations
are very close, but not equal to those of $f_1$. This is possible by
the work of Farb and Mosher on Schottky subgroups of $MCG$
\cite{FM}. Furthermore, one can arrange that nontrivial powers of $g$
are not conjugate to powers of $f_i$'s (that's automatic once the
(un)stable laminations of $g$ are sufficiently close to those of $f_1$).
It then follows that for small $i$ $g$ has positive translation length
in $\mathcal X_i$ and for large $i$ its orbits are bounded.
\end{remark}

\subsection{WPD}

For the construction of quasi-homomorphisms on groups acting
isometrically on hyperbolic complexes it is important to have Weak
Proper Discontinuity of the action \cite{BF}.

\begin{prop}\label{wpd.surface}
  The elements $f_1,\cdots,f_k$ chosen at the start of the
  construction satisfy WPD: For every $i=1,2,\cdots,k$, every $x\in \X$,
  and every $C>0$ there is $N>0$ such that
$$\{g\in MCG| d(x,g(x))\leq C, d(f_i^N(x),gf_i^N(x))\leq C\}$$
is finite.
\end{prop}

We will omit the proof, since it is easier than the corresponding
statement for $Out(F_n)$, which we prove in Section \ref{WPD}.

\section{A hyperbolic complex for $Out(F_n)$}

Recall that $f\in Out(F_n)$ is {\it fully irreducible} if for all
proper free factors $F$ of $F_n$ and all $k>0$ we have that $f^k(F)$
is not conjugate to $F$. For convenience, we restate our main result.

\begin{main}
  For any finite collection $f_1,\cdots,f_k$ of fully irreducible
  elements of $Out(F_n)$ there is a connected $\delta$-hyperbolic
  graph $\X$ equipped with an (isometric) action of $Out(F_n)$ such
  that
\begin{itemize}
\item the stabilizer in $Out(F_n)$ of a simplicial tree in
  $\overline{\mathcal{PT}}$ has bounded
  orbits, 
\item the stabilizer in $Out(F_n)$ of a proper free factor $F\subset F_n$
  has bounded orbits, and
\item
  $f_1,\cdots,f_k$ have nonzero translation lengths.
\end{itemize}
\end{main}

We will start with some preliminaries. By $\mathcal T=\mathcal T_n$
denote the space of free cocompact simplicial metric $F_n$-trees
without vertices of valence 1 and 2. If $\gamma$ is a conjugacy class in $F_n$
and $T\in \mathcal T$, denote by $\langle T,\gamma\rangle$ the
translation length of $\gamma$ in $T$. The group $Out(F_n)$ acts on the
right on $\mathcal T$ by the ``change of marking'', i.e. by the rule
that $\langle Tg,\gamma\rangle=\langle T,g(\gamma)\rangle$.  The group
$\R^+$ acts on $\mathcal T$ by scaling and this action commutes with
the action of $Out(F_n)$. The projectivized space $\mathcal
{PT}=\mathcal {PT}_n=\mathcal T/\R^+$ is Culler-Vogtmann's Outer space
\cite{CV}. By $\overline{\mathcal T}$ denote the closure of $\mathcal
T$ in the space of minimal $F_n$-trees. Both $Out(F_n)$ and $\R^+$
continue to act on $\overline{\mathcal T}$; let $\overline{\mathcal
  {PT}}$ be the projectivization of $\overline{\mathcal T}$. This is
Culler-Morgan's equivariant compactification of Outer space \cite{CM}.

To every fully irreducible outer automorphism $f$ one associates the
stable tree $T^+_f$ and the unstable tree $T^-_f$. In
$\overline{\mathcal {PT}}$ they are defined as limits
$T_f^+=\lim_{k\to\infty}T_0f^k$ and $T_f^-=\lim_{k\to\infty}T_0f^{-k}$ for
any tree $T_0$ in Outer space. In $\overline{\mathcal {T}}$ they are
defined only up to scale, but in a similar way after choosing the
right scaling factors. More precisely
$$T_f^+=\lim_{k\to\infty} T_0f^k/\lambda^k\qquad\mbox{and}\qquad
T_f^-=\lim_{k\to \infty} T_0f^{-k}/\mu^k$$ where $\lambda$ is the
growth rate of $f$ and $\mu$ the growth rate of $f^{-1}$ (see
below). These trees satisfy $T_f^+f=\lambda T_f^+$ and
$T_f^-f=T_f^-/\mu$. The following important fact was proved by Levitt
and Lustig \cite{LL}.

\begin{prop}\label{north-south}
The fully irreducible automorphism $f$ acts on $\overline{\mathcal
  {PT}}$ with north-south dynamics: $T_f^\pm$ are the only fixed points,
and any compact set that does not contain $T_f^-$ [$T_f^+$] converges
uniformly under iteration by $f$ [$f^{-1}$] to $T_f^+$ [$T_f^-$].
\end{prop} 

For convenience, we will say that a tree $T$ is an {\it irreducible
  tree} if $T=T^+_f$ for some fully irreducible
automorphism $f$.

\subsection{Some train track facts}

Recall that a fully irreducible automorphism is {\it geometric} if it
is induced by a homeomorphism of a compact surface with (necessarily
connected) boundary;
otherwise it is {\it non-geometric}. A
fully irreducible automorphism is geometric if and only if it has a
nontrivial periodic conjugacy class (which is necessarily either fixed
or sent to its inverse) \cite{BH}. A fully irreducible
automorphism is non-geometric if and only if the associated stable
tree is free (i.e. every nontrivial element has nonzero translation
length).

In this section we generalize some of the lemmas from \cite{gafa}. In
that paper we proved, for example, that the action of $F_n$ on the
product $T_f^+\times T_f^-$ of the stable and the unstable tree
of a fully irreducible automorphism is discrete. The case of a
geometric $f$ is classical, and we focused our attention on
nongeometric $f$. Here we are interested in the action of $F_n$ on
the product $T_1\times T_2$ of two irreducible trees, associated 
with two possibly unrelated automorphisms. The proofs in this more general
setting are only slight variations of the original. 

Recall that a map $\rho:H\to H$ on a finite graph without valence 1
vertices is a {\it train track map} if it sends vertices to vertices
and for every $i>0$ the map $\rho^i$ restricted to any edge is locally
injective. Such a map is a {\it topological representative} of some
$f\in Out(F_n)$ if after a suitable identification (called {\it marking})
$\pi_1(H)\cong F_n$ the map $\rho$ induces $f$ in $\pi_1$. Every fully
irreducible automorphism $f$ admits a train track representative $\rho$
\cite{BH}. Up to scale, there is a unique assignment of lengths to the
edges of $H$ and a constant $\lambda$ (the {\it growth rate} of $f$)
so that for every edge $e$ we have $length(\rho(e))=\lambda~ length(e)$.

Replace $\rho$ by a power if necessary so that there is a fixed point
$x$ in the interior of some edge. Let $I$ be an
$\epsilon$-neighborhood around $x$ so that $\rho(I)\supset I$ (and the
orientation is preserved). Choose an isometry
$\ell:(-\epsilon,\epsilon)\to I$ and extend it uniquely to a locally
isometric immersion $\ell:\R\to H$ such that
$\ell(\lambda^Nt)=\rho^N(\ell(t))$.  A {\it stable leaf segment} is
the restriction of $\ell$ to a finite segment (possibly
reparameterized).  The collection of stable leaf segments does not
depend on the choice of $x$ and $I$. One can talk about stable leaf
segments with respect to a different graph $H'$ representing $F_n$: if
$\tau:H\to H'$ is a given homotopy equivalence, let $[\tau\ell]$ be
the induced line in $H'$ pulled tight, and then consider finite
subsegments of this line. The collection of these segments does not
depend on the choice of the train track representative $\rho:H\to H$.

Likewise, {\it unstable leaf segments} are stable leaf
segments for $f^{-1}$. 

An edge-path $p$ in $H$ is {\it legal} if $\rho^ip$ is locally injective
for all $i=0,1,\cdots$. For example, edges and stable leaf segments
(those that are also edge-paths)
are legal. If an immersed edge-path has the form $p=u\cdot v\cdot w$ where $v$
is a legal segment, then
$\rho(p)=\rho(u)\rho(v)\rho(w)$. Now if $v$ is sufficiently long, say $|v|>C$,
then after canceling against $\rho(u)$ and $\rho(w)$ what is left of $\rho(v)$
will still be longer than $v$. Such a constant $C$ is called a {\it
  critical constant}.

Let $\rho:H\to H$ be a train track map representing a fully irreducible
automorphism $f$. An immersed line $\ell$ in $H$ will be called {\it
  bad} if the legal segments of tightened iterates $[\rho^i(\ell)]$ have
uniformly bounded size for $i=0,1,2,3,\cdots$. The uniform bound can
be taken to be a critical constant.

Examples of bad lines are: unstable leaves,
$\cdots\gamma\gamma\gamma\cdots$ where $\gamma$ is periodic, as well
as concatenations of unstable half-lines with fixed endpoints
(possibly with
powers of $\gamma$ inserted between them).

\begin{lemma}\label{bad line}
A bad line $\ell$ either contains arbitrarily long unstable leaf segments or
for every $K$ there is $N$ such that $\rho^N(\ell)$ contains a segment of
length $\geq K$ representing a periodic conjugacy class.
\end{lemma}

\begin{proof}
  Fix $C>0$. Let $\rho':H'\to H'$ be a train track representative for
  $f^{-1}$ and let $\tau:H\to H'$ be the difference of markings. If
  $f$ is nongeometric we may apply \cite[Lemma 2.10]{gafa}. It says
  that for any $C'>0$ there is $N_0$ such that for any immersed line
  $\ell'$ in $H$, either $\rho^{N_0}(\ell')$ contains a stable leaf
  segment of length $>C'$ or ${\rho'}^{N_0}\tau\ell'$ contains a
  stable (for $f^{-1}$) leaf segment of length $>C'$. Apply this lemma
  to the line $\ell'=\rho^{N_0}(\ell)$ to conclude that when $\tau\ell$ is
  transferred to $H'$ it contains a legal segment of length $>C'$.
  When $C'$ is big enough, this means that $\ell$ contains a long
  unstable leaf segment.

The proof of Lemma 2.10 in \cite{gafa} holds also for geometric $f$
provided that there is a uniform bound on the length of a Nielsen
segment in all iterates $[\rho^i(\ell)]$, $i=0,1,2,\cdots$. If there is
no such bound, the statement trivially holds.
\end{proof}

\begin{cor}\label{bad line+}
Let $\ell$ be a bad line and suppose that there is a uniform bound on
the length of periodic segments in the iterates of $\ell$. Then for
every $A$ there is $B$ such that every segment in $\ell$ of length $B$
contains an unstable leaf segment of length $A$.
\end{cor}

\begin{proof}
If this is false, then there is a sequence of longer and longer
segments in $\ell$ that don't contain unstable leaf segments of length
$A$. Passing to a subsequence and taking a limit produces a line that
violates Lemma \ref{bad line}.
\end{proof}

\begin{lemma}\label{2.10}
  Let $\rho:H\to H$ and $\rho':H'\to H'$ be train track
  representatives of two fully irreducible automorphisms $f$ and $f'$
  and let $\tau:H\to H'$ be a homotopy equivalence representing the
  difference of markings. Assume $\tau\rho^k\not\simeq\rho'^l\tau$ for
  all $k,l>0$ (equivalently, the (un)stable trees for $f,f'$ are
  distinct).

Then for every $C>0$ there are $N_0,L_0>0$ such that if $\iota$ is an
immersion of a line, a circle of length $\geq L_0$, or a closed
interval of length $\geq L_0$ and $\iota'$ is obtained from
$\tau\iota$ by pulling tight, then either
\begin{enumerate}[(A)]
\item $[\rho^{N_0}\iota]$ contains a legal segment of length $>C$, or
\item $[\rho'^{N_0}\iota']$ contains a legal segment of length $>C$,
  or
\item $[\rho^N\iota]$ contains a $\rho$-periodic segment of
  length $>C$ for some $N>0$, or
\item $[\rho'^N\iota']$ contains a $\rho'$-periodic segment of length
  $>C$ for some $N>0$.
\end{enumerate}
\end{lemma}

\begin{proof}
It suffices to prove the statement for lines; the other cases follow
by taking limits.

Suppose the statement is false. Then we have a sequence of immersed
lines that fail (A)-(D) with $N_0\to\infty$. Denote by $\iota$ a
limiting line of this sequence. Then $\iota$ is a bad line with
respect to both $\rho$ and $\rho'$ and 
satisfies Corollary \ref{bad line+} for both. 
We conclude that
long unstable leaf segments for $\rho$ contain long unstable leaf
segments for $\rho'$ and {\it vice versa}.
Then \cite[Theorem
  2.14]{gafa} implies that $f$ and $f'$ have common positive powers,
contradiction. (In the language of \cite{gafa} we have shown that $f$
and $f'$ have the same unstable laminations.)
\end{proof}

Recall that the {\it legality} of an immersed loop $\alpha$ in $H$ is
the ratio
\begin{multline*}
LEG_H(\alpha)=\\
\frac{\mbox{sum of the lengths of maximal legal leaf segments of
  }\alpha\mbox{ of length }\geq C}{\mbox{length}(\alpha)}
\end{multline*}
where $C>0$ is a sufficiently big constant (for example, bigger than
a critical constant).

The next lemma is a variation of \cite[Lemma 5.5]{gafa} and we omit
the proof.

\begin{lemma}\label{5.5}
\begin{enumerate}[(1)]
\item For every $\epsilon>0$ and $A>0$ there is $N_3=N_3(\epsilon,A)$
  such that if $LEG_H(\alpha)\geq\epsilon$ then
$$length[\rho^N(\alpha)]\geq A~ length(\alpha)$$ for all $N\geq N_3$.
\item For every $\epsilon>0$ there is $\delta>0$ such that if
  $LEG_H(\alpha)\geq\epsilon$ then $\langle T^+_f,\alpha\rangle\geq\delta~
  length_H(\alpha)$. 
\item For every $\epsilon>0$ and every $L>0$ there is
  $N_4=N_4(\epsilon,L)>0$ such that if $LEG_H(\alpha)\geq\epsilon$
  then for all $N\geq N_4$ the set of points of $[\rho^N(\alpha)]$
  whose $L$-neighborhood is a stable leaf segment has
  total length $\geq (1-\epsilon)length[\rho^N(\alpha)]$. 
\end{enumerate}
\end{lemma}

The following generalizes \cite[Lemma 5.6]{gafa}. We say that a
conjugacy class $\alpha$ is {\it primitive} if any of its elements can
be extended to a basis of $F_n$.

\begin{lemma}\label{5.6}
  Let $f,f'$ be fully irreducible automorphisms with $T_f^+\neq
  T_{f'}^+$ (projectively), and therefore $T_f^-\neq T_{f'}^-$. Let
  $\rho:H\to H$, $\rho':H'\to H'$ and $N_0$ be as in Lemma \ref{2.10}
  and let $\alpha$ be a conjugacy class.  Assume either that $f,f'$
  are nongeometric or that $\alpha$ is a primitive conjugacy
  class. Then there is $\epsilon>0$ such that for every $N\geq N_0$
  either $LEG_H(\rho^N(\alpha)_H)\geq\epsilon$ or
  $LEG_{H'}(\rho'^N(\alpha)_{H'})\geq\epsilon$.
\end{lemma}

\begin{proof}
We first argue that (C) and (D) of Lemma \ref{2.10} cannot occur when
applied to $\alpha$. This
is clear if $f$ and $f'$ are non-geometric, so assume that $\alpha$ is
primitive. We now use an argument of Yael Algom-Kfir \cite{yael}. The
loop $\rho^N(\alpha)$ also represents a primitive element, while the loop
representing the indivisible fixed class $\gamma$ crosses every edge
of $H$ twice. This is true after collapsing a maximal forest in $H$ as
well, and the Whitehead graph of $\gamma$ in the resulting rose is a
circle that passes through every vertex. It follows that any loop that
contains two consecutive copies of $\gamma$ will have Whitehead graph
that contains this circle, and hence it does not have a cut point. But
it is a classical theorem of Whitehead \cite{wh} that the Whitehead
graph of a primitive class is either disconnected
or contains a cut vertex. Thus $[f^N(\alpha)]$ cannot contain two
consecutive copies of $\gamma$. This finishes the proof that (C) and
(D) cannot occur. 

The rest of the argument is identical to the proof of Lemma 5.6 in
\cite{gafa}, using Lemma \ref{2.10} in place of Lemma 2.10 of
\cite{gafa}. 
\end{proof}

By $|\alpha|$ denote the length of the conjugacy class $\alpha$ with
respect to a
fixed graph.

\begin{cor}\label{T2}
  Let $f,g$ be fully irreducible automorphisms and assume $T_f^+\neq
  T_g^+$ (equivalently, $T_f^-\neq T_g^-$).  There is $\delta>0$
  such that for all primitive conjugacy classes $\alpha$ we have
  either $\langle T_f^+,\alpha\rangle \geq\delta |\alpha|$ or $\langle
  T_g^+,\alpha\rangle\geq\delta |\alpha|$. In particular, for any
  $C>0$ there are only finitely many primitive conjugacy classes
  $\alpha$ with both $\langle T_f^+,\alpha\rangle<C$ and $\langle
  T_g^+,\alpha\rangle <C$.

If $f$ and $g$ are nongeometric then these statements hold for all
conjugacy classes.
\end{cor}

\begin{proof}
This follows from Lemma \ref{5.5}(2) and Lemma \ref{5.6}.
\end{proof}

\begin{remark}\label{comparable lengths}
  Note that there is $M>0$ such that $\langle
  T_1,\alpha\rangle+\langle T_2,\alpha\rangle\leq M |\alpha|$ for any
  conjugacy class $\alpha$. This is because there is an equivariant
  Lipschitz map $T\to T_i$, $i=1,2$ from any tree $T$ in Outer
  space. Therefore, $$\langle T_f, \alpha\rangle +\langle
  T_g,\alpha\rangle \sim |\alpha|$$ for primitive $\alpha$ (or all
  $\alpha$ if $f,g$ are nongeometric) in the sense that the ratio is
  bounded away from 0 and $\infty$.
\end{remark}

If a sequence $T_i$ in $\overline{\mathcal T}$ converges projectively
then we say that a sequence $\lambda_i>0$ is a {\it scaling sequence
  for $T_i$} if $T_i/\lambda_i$ converges in $\overline{\mathcal T_i}$
(without further scaling). If one scaling sequence for $T_i$ converges
to infinity then they all do. For a sequence of distinct $g_i$, one
expects a scaling sequence for $Tg_i$ to converge to infinity, but
note that $\mu^i T^-_{f}f^i=T_f^-$, so the scaling sequence is
$1/\mu^i$. The following Proposition is a weak converse to this.

\begin{prop}\label{T1}
Assume $p_0\neq q_0$ are irreducible trees, $p_i=p_0\cdot g_i$,
$q_i=q_0\cdot g_i$ for a sequence of distinct $g_i\in Out(F_n)$. Also
assume that $p_i$ and $q_i$ converge projectively. Then a scaling
sequence for either $p_i$ or $q_i$ converges to infinity.
\end{prop}

\begin{proof}
Suppose $p_i/\lambda_i\to p$ and $q_i/\mu_i\to q$. Let $\alpha$ be a
primitive conjugacy class in $F_n$.  Therefore,
$$\langle p_0g_i/\lambda_i,\alpha\rangle\to \langle p,\alpha\rangle $$ and hence
$$\langle p_0,g_i(\alpha)\rangle\overset\sim < \lambda_i$$
and similarly $\langle q_0,g_i(\alpha)\rangle\overset\sim <\mu_i$.

Now suppose that both $\lambda_i$ and $\mu_i$ are bounded. 
By Corollary \ref{T2} there are only finitely
many possibilities for $g_i(\alpha)$. Now apply this to the primitive conjugacy
classes of elements in $F_n$ of word length $\leq 2$. Since an
automorphism that fixes these conjugacy classes is necessarily inner
(a standard fact),
it follows that there are only finitely many choices for $g_i$,
contradiction. 
\end{proof}

\subsection{Measured geodesic currents}\label{currents}

Measured geodesic currents (or just currents in the sequel) were
introduced by Francis Bonahon, first on hyperbolic surfaces
\cite{bonahon1} in order to study the geometry of Teichm\"uller space,
and later in the setting of any word-hyperbolic group
\cite{bonahon2}. Of interest for this paper is the case of free
groups, further studied by Reiner Martin in his thesis \cite{reiner},
and more recently by Ilya Kapovich, Martin Lustig and others (see
\cite{ilya-lustig} and references therein). Martin's thesis has never
been published, but most of his results are available in
\cite{ilya-survey}. 

Let $\partial F_n$ denote the Cantor set of ends of $F_n$ and let
$\partial^2F_n=(\partial F_n\times\partial F_n-\Delta)/\Z_2$ be the
space of unordered pairs of distinct points of $\partial F_n$ (thought
of as the space of unoriented biinfinite geodesics in $F_n$). By
$\mathcal C(F_n)$ denote the collection of compact open subsets of
$\partial^2F_n$. A {\it current} $\eta$ is an additive function
$\mathcal C(F_n)\to [0,\infty)$ which is invariant under the
  (diagonal) action of $F_n$ (``additive'' means that $\eta(C_1\sqcup
  C_2)=\eta(C_1)+\eta(C_2)$). The space $\mathcal {MC}(F_n)$ of
  currents has the structure of the cone (positive linear combinations
  of currents are currents) and a natural topology, as a subset of
  $[0,\infty)^{\mathcal C(F_n)}$. Projectivizing
  gives a compact space $\mathcal {PMC}(F_n)$ of projectivized
  (measured geodesic) currents.

For each indivisible conjugacy class $\gamma$ in $F_n$ we define a
current $\eta_\gamma$ induced by $\gamma$: if $C\in \mathcal C(F_n)$
then $\eta_\gamma(C)$ is the number of lifts of $\gamma$ which are in
$C$. If $\gamma=\beta^k$ with $\beta$ indivisible and $k>0$, define
$\eta_\gamma=k\eta_\beta$. Thus the set of conjugacy classes in $F_n$
can be viewed as a subset of $\mathcal {MC}(F_n)$, and their image in
$\mathcal {PMC}(F_n)$ is dense \cite{reiner}. The group $Out(F_n)$
acts on the space of currents via $$g(\eta)(C)=\eta(g^{-1}(C))$$ This
action extends the action on conjugacy classes, in the sense that
$g(\eta_\gamma)=\eta_{g(\gamma)}$.

For a fully irreducible automorphism $f$ one can define the
{\it stable current} $\Upsilon^+_f$ and the {\it unstable current}
$\Upsilon^-_f$. Projectively they can be defined as
$\Upsilon_f^+=\lim_{k\to\infty} f^k(\eta_\gamma)$ and
$\Upsilon_f^-=\lim_{k\to\infty} f^{-k}(\eta_\gamma)$ for any
primitive conjugacy class $\gamma$ in $F_n$ (or indeed any
non-periodic conjugacy class). In
$\mathcal{MC}(F_n)$ the stable and unstable currents are defined only
up to scale:
$$\Upsilon_f^+=\lim_{k\to\infty}
f^{k}(\eta_\gamma)/\lambda^k\qquad\mbox{and}\qquad 
\Upsilon_f^-=\lim_{k\to\infty}
f^{-k}(\eta_\gamma)/\mu^k$$
where $\lambda$ and $\mu$ are the growth rates of $f$ and $f^{-1}$.

The following important fact was proved by Martin \cite{reiner}.

\begin{prop}\label{reiner1}
Every non-geometric fully irreducible automorphism $f$ acts on
$\mathcal{PMC}(F_n)$ with north-south dynamics: $\Upsilon_f^\pm$ are
the only two fixed points and every compact set that does not contain
$\Upsilon_f^-$ [$\Upsilon_f^+$] uniformly converges to $\Upsilon_f^+$
[$\Upsilon_f^-$] under iteration by $f$ [$f^{-1}$].
\end{prop}

Of course, this result is false for geometric automorphisms since the
current representing the boundary is fixed as well. However,
Martin also observed that the above theorem holds for geometric
automorphisms as well, provided one restricts to a certain closed
invariant 
subset $\mathcal M(F_n)$ in $\mathcal{PMC}(F_n)$. This set is defined
as the closure of the set of projectivized currents of the form $\eta_\gamma$
where $\gamma$ is a primitive conjugacy class. Thus
$\mathcal M(F_n)$ contains all currents of the form
$\Upsilon_f^\pm$. It is also known that for $n\geq 3$ $\mathcal
M(F_n)$ is the unique minimal nonempty closed $Out(F_n)$-invariant subset of
$\mathcal{PMC}(F_n)$ \cite{ilya-lustig2}.

\begin{prop}\label{reiner2}
Every fully irreducible automorphism
$f$ acts on $\mathcal M(F_n)$ with north-south dynamics:
$\Upsilon_f^\pm$ are the only fixed points and every compact subset of
$\mathcal M(F_n)$ that does not contain $\Upsilon_f^-$ [$\Upsilon_f^+$]
converges uniformly to $\Upsilon_f^+$ [$\Upsilon_f^-$] under iteration
by $f$ [$f^{-1}$].
\end{prop}

\begin{proof}[Proofs of Propositions \ref{reiner1} and \ref{reiner2}]
  Let $\rho:H\to H$ be a train track representative for $f$ and let
  $\ell$ be a stable leaf. A typical compact and open set
  $C\subset \partial^2F_n$ is determined by a finite edge path in the
  universal cover $\tilde H$ -- it consists of all lines that contain
  this path. So one can view a current $\eta$ as assigning a number to such
  an edge path. Equivariance dictates that translates be assigned the
  same number, thus $\eta$ assigns numbers to edge paths in
  $H$. Additivity then translates to saying that
  $\eta(\pi)=\sum\eta(\pi_i)$ as $\pi_i$ range over all 1-edge
  extensions of $\pi$. The current $\Upsilon^+_f$ assigns 0 to edge
  paths that are not crossed by $\ell$, and more generally it assigns the
  frequency of occurrence of this path in $\ell$. Lemma \ref{5.5}(3)
  implies that all conjugacy classes $\alpha$ with
  $LEG_H(\alpha)\geq\epsilon$ converge to $\Upsilon_f^+$
  uniformly. The same statement holds when $f$ is replaced by $f^{-1}$
  and 
  $H$ by a train track graph for $f^{-1}$. Lemma \ref{5.6} now
  implies that every primitive conjugacy class $\alpha$ (or any
  non-trivial class if $f$ is non-geometric) either converges
  uniformly to $\Upsilon^+_f$ under forward iteration, or to
  $\Upsilon^-_f$ under backward iteration. Since conjugacy classes are
  dense in $\mathcal{PMC}(F_n)$ and primitive conjugacy classes are
  dense in $\mathcal M(F_n)$, both propositions follow.
\end{proof}

In \cite{ilya-lustig} I. Kapovich and Lustig extended the length
pairing between trees and conjugacy classes to trees and
currents. More precisely, they proved the following.

\begin{prop}
There is a length pairing $\langle
\cdot,\cdot\rangle:\overline{\mathcal T}\times \mathcal{MC}(F_n)\to
     [0,\infty)$ satisfying:
\begin{itemize}
\item it extends the usual length pairing, i.e. $\langle
  T,\eta_\gamma\rangle =\langle T,\gamma\rangle$ for any conjugacy
  class $\gamma$,
\item $\langle Tg,\eta\rangle=\langle T,g(\eta)\rangle$,
\item it is homogeneous in the first coordinate, i.e. $$\langle
  \lambda T,\eta\rangle=\lambda \langle T,\eta\rangle$$ for
  $\lambda> 0$,
\item it is linear in the second coordinate, i.e. $$\langle
  T,\lambda_1\eta_1+\lambda_2\eta_2\rangle=\lambda_1\langle
  T,\eta_1\rangle +\lambda_2\langle
  T,\eta_2\rangle$$ for $\lambda_1,\lambda_2\geq 0$, and
\item it is continuous.
\end{itemize}
\end{prop}

The following statements are easy consequences of the above.

\begin{cor} \label{easy}
Let $f$ be any fully irreducible
  automorphism, $T\in\overline{\mathcal T}$,
  $\Upsilon\in\mathcal{MC}(F_n)$. Then
\begin{enumerate}[(1)]
\item $\langle T_f^\pm,\Upsilon_f^\mp\rangle=0$.
\item Assume either that $f$ is non-geometric or that
  $\Upsilon\in\mathcal M(F_n)$. If $\langle T_f^\pm,\Upsilon\rangle=0$
  then $\Upsilon=\Upsilon_f^\mp$ (projectively).
\item If $\langle T,\Upsilon_f^\pm\rangle=0$ then $T=T_f^\mp$ (projectively).
\end{enumerate}
\end{cor}

\begin{proof}
(1) $\langle T_f^+,\Upsilon_f^-\rangle=\langle \lim
  T_0f^i/\lambda^i,\lim f^{-i}(\eta_\gamma)/\mu^i\rangle =\lim \langle
  T_0,\eta_\gamma\rangle/(\lambda^i\mu^i)=0$.

(2) Let $\gamma$ be a primitive conjugacy class. We start by observing
  that $\langle T_f^+,\Upsilon_f^+\rangle=\lim \langle
  T_0f^i/\lambda^i,f^i(\eta_\gamma)/\lambda^i\rangle=\lim\langle
  T_0f^{2i}/\lambda^{2i},\gamma\rangle=\langle
  T_f^+,\gamma\rangle>0$. If we had $\langle T_f^+,\Upsilon\rangle=0$
  for some $\Upsilon\neq\Upsilon_f^-$ then
  $f^i(\Upsilon)/\lambda_i\to\Upsilon_f^+$ for a suitable scaling
  sequence $\lambda_i$ (actually, one can take
  $\lambda_i=\lambda^i$), and by continuity we would conclude $\langle
  T_f^+,\Upsilon_f^+\rangle=0$, contradiction.

(3) Similar to (2).
\end{proof}

When $T$ is an irreducible tree, denote by $T^*$ the dual current, i.e.  if
$T=T^+_f$ then $T^*=\Upsilon^-_f$. (This is well defined since if
$T_f^+=T_g^+$ then $f^m=g^k$ for some $m,k>0$ and therefore
$\Upsilon_f^\pm=\Upsilon_g^\pm$.) Thus $\langle T,T^*\rangle=0$. The
current $T^*$ is defined only up to scale. Also note that
$(Tf)^*=f^{-1}(T^*)$ and that $T^*$ is the only current in $\mathcal
M(F_n)$ whose length in $T$ is 0.

\begin{lemma}\label{dual convergence}
Let $T_i,T$ be irreducible trees. Then $T_i\to T$ iff $T_i^*\to T^*$
(projectively). 
\end{lemma}

\begin{proof}
Say $T_i/\lambda_i\to S$ and $T_i^*/\mu_i\to \Upsilon$ (without
scaling). Then $\langle S,\Upsilon\rangle = \langle \lim
T_i/\lambda_i,\lim T_i^*/\mu_i\rangle = 0$ so
if $S=T$ then $\Upsilon=T^*$ (note that $\Upsilon\in\mathcal M(F_n)$)
and if $\Upsilon=T^*$ then $S=T$.
\end{proof} 

\begin{lemma}\label{scaling}
Let $T$ be an irreducible tree. Suppose trees $Tg_i$ converge
projectively to a tree $\not= T$. Suppose also that $g_i(T^*)$
converges projectively to a current $\not=T^*$ (or, equivalently by
Lemma \ref{dual convergence}, $Tg^{-1}_i$ converges projectively to a
tree $\neq T$). Then a scaling sequence for $Tg_i$ is also a scaling
sequence for $g_i(T^*)$.
\end{lemma}

\begin{proof}
Suppose $Tg_i/\lambda_i\to T'$.
We have $\langle T,g_i(T^*)/\lambda_i\rangle=\langle
Tg_i/\lambda_i,T^*\rangle \to \langle T',T^*\rangle>0$, so $\lambda_i$
is a scaling sequence for $g_i(T^*)$.
\end{proof}

\begin{lemma}\label{doubly bad}
Suppose $a\not=b$ are two irreducible trees and that $ag_i$ and $bg_i$
converge projectively to $a$ and $b$ respectively. Also assume that
$g_i(a^*)$ converges projectively to a current $\neq b^*$ and
$g_i(b^*)$ converges projectively to a current $\neq a^*$
(equivalently, $ag_i^{-1}$ converges projectively to a tree $\neq b$
and $bg_i^{-1}$ to a tree $\neq a$).  Then a scaling sequence for
$ag_i$ is a scaling sequence for $g_i(b^*)$ and a scaling sequence for
$bg_i$ is a scaling sequence for $g_i(a^*)$.
\end{lemma}

\begin{proof}
Suppose $ag_i/\lambda_i\to a$. We have
$$\langle a,g_i(b^*)/\lambda_i\rangle=\langle
  ag_i/\lambda_i,b^*\rangle \to \langle a,b^*\rangle >0$$
which means that $\lambda_i$ is a scaling sequence for
$g_i(b^*)$. The claim about $g_i(a^*)$ is similar.
\end{proof}

\subsection{Verification of (A1) and (A2)}

Fix a finite collection of fully irreducible automorphisms
$f_1,\cdots,f_k$. Choose small closed neighborhoods $D^\pm_i$ of
$T^\pm_{f_i}$ determining annuli $A_i=(D_i^-,D_i^+)$ and consider the
corresponding annulus system $\mathcal A=\{\pm A_ig|g\in Out(F_n),
i=1,\cdots,k\}$ consisting of all translates of these. For notational
simplicity we will assume $k=1$, $f=f_1$ and $D^\pm=D_1^\pm$.

\begin{lemma}\label{OutA1}
If $D^\pm$ are chosen small enough, the following holds.
If $a,b,c,d$ are irreducible trees and $a\neq b$, $c\neq d$, then
$(ab|cd)<\infty$. 
\end{lemma}

\begin{proof}
By translating, we may assume that $a,b,c,d$ are outside $D^\pm$.  If
$(ab|cd)=\infty$, then there are infinitely many distinct $g_i\in
Out(F_n)$ so that $ag_i^{-1},bg_i^{-1}\in D^-$ and
$cg_i^{-1},dg_i^{-1}\in D^+$ (or switch $D^-$ and $D^+$). We may
assume that these sequences converge projectively. Let
$\alpha_i,\beta_i,\gamma_i,\delta_i$ be scaling sequences for
$ag_i^{-1},bg_i^{-1},cg_i^{-1},dg_i^{-1}$ respectively, so
e.g. $ag_i^{-1}/\alpha_i$ converges in $\overline{\mathcal
  T}$. Likewise, let $\alpha_i',\beta_i',\gamma_i',\delta_i'$ be
scaling sequences for $ag_i,bg_i,cg_i,dg_i$ respectively. By
Proposition \ref{T1} at least three of
$\alpha_i,\beta_i,\gamma_i,\delta_i$ go to $\infty$, and we assume
$\alpha_i,\beta_i,\gamma_i\to\infty$. Likewise three of
$\alpha_i',\beta_i',\gamma_i',\delta_i'$ go to $\infty$, say
$\alpha_i',\beta_i',\delta_i'$ (the other possibilities are similar).

Now we have the following cases.

{\bf Case 1.} $bg_i\to T_b\neq b$ (projectively). Then by Lemma
\ref{scaling} $\beta_i'$ is a scaling sequence for $g_i(b_i^*)$ (note
that $bg_i^{-1}\in D^-$ so cannot converge to $b$). If we let
$T_a=\lim ag_i^{-1}/\alpha_i$ and $\Upsilon_b=\lim g_i(b^*)/\beta_i'$
then
$$\langle ag_i^{-1}/\alpha_i,g_i(b^*)/\beta'_i\rangle=\langle
a,b^*\rangle/(\alpha_i\beta'_i)\to 0$$ so $\langle
T_a,\Upsilon_b\rangle=0$. Likewise $\langle T_c,\Upsilon_b\rangle=0$
where $T_c=\lim cg_i^{-1}/\gamma_i$. But that's a contradiction --
there is no current in $\mathcal M(F_n)$ that has length 0 in trees
close to both $T_f^+$ and $T_f^-$. (Note that $T_a$ is close to
$T_f^-$ and $T_c$ to $T_f^+$.) Indeed, a limiting argument would
produce a current in $\mathcal M(F_n)$ whose length is 0 in both
$T^+_f$ and $T^-_f$, violating Corollary \ref{easy}.

{\bf Case 2.} $ag_i\to T_a\neq a$. This is the same as Case 1 after
exchanging the roles of $a$ and $b$.

{\bf Case 3.} $ag_i\to a$ and $bg_i\to b$. Then by Lemma \ref{doubly
  bad} scaling sequences for $g_i(a^*)$ and $g_i(b^*)$ are
$\beta'_i$ and $\alpha'_i$ (note that $ag_i^{-1},bg_i^{-1}\in D^-$ so
neither can converge to $a$ or $b$) and they also go to $\infty$, so
the same argument as in Case 1 holds.
\end{proof}

\begin{lemma}\label{OutA2}
If $D^\pm$ are small enough then for any irreducible trees $a,b,c,d$
we have either $(ac|bd)=0$ or $(ad|bc)=0$.
\end{lemma}

\begin{proof}
Suppose $(ac|bd)>0$ and $(ad|bc)>0$. After replacing $a,b,c,d$ by
(simultaneous) translates if necessary, we may assume that $a,c\in
D^-$ and $b,d\in D^+$ (or interchange $D^-$ and $D^+$). Now there are
two cases.

{\bf Case 1.} There is $g\in Out(F_n)$ such that
$ag,dg\in D^-$, $bg,cg\in D^+$.

Consider the expression
$$\frac{\langle a,c^*\rangle\langle b,d^*\rangle}{\langle
  a,d^*\rangle\langle b,c^*\rangle}$$
This expression does not change after scaling $a,b,c^*,d^*$, and it
does not change after applying $g$, i.e. replacing $a,b$ by $ag,bg$
and $c^*,d^*$ by $g^{-1}(c^*),g^{-1}(d^*)$. 

Also note that when $c_i\to T$ (an irreducible tree) then $c_i^*\to T^*$ so we
may assume that $a$ is close to $T_f^-$, $c^*$ to $(T_f^-)^*$, $b$ to
$T_f^+$, and $d^*$ to $(T_f^+)^*$. By the continuity of the pairing, the
expression above is small (both numbers in the numerator are close to
0, the numbers in the denominator are close to $\langle
T_f^-,(T_f^+)^*\rangle>0$ and $\langle T_f^+,(T_f^-)^*\rangle>0$ ). After
applying $g$, the expression is close to $\infty$ (the numbers in the
numerator are close to $\langle T_f^-,(T_f^+)^*\rangle>0$ and $\langle
T_f^+,(T_f^-)^*\rangle>0$ and both numbers in the denominator are close to
0). Contradiction.

{\bf Case 2.} There is $g\in Out(F_n)$ such that
$ag,dg\in D^+$, $bg,cg\in D^-$. The argument is similar.
\end{proof}

\begin{proof}[Proof of the Main Theorem.]
Let $M\subset\overline{\mathcal{PT}}$ 
consisting of all irreducible trees. For the annulus system take the
restriction to $M$ of the annulus system considered above.
Thus (A1) follows from Lemma
\ref{OutA1} and (A2) from Lemma \ref{OutA2}.
The resulting Bowditch complex $\X$ is hyperbolic. The statements
about orbits and translation lengths are verified in the next section.
\end{proof}

\subsection{Orbits}

\begin{prop}
The elements $f_1,\cdots,f_k$ chosen at the start of the construction
have nonzero translation lengths.
\end{prop}

\begin{proof}
By the north-south dynamics we see that
$d(x,xf_i^N)\to\infty$, in fact $\lim\inf\frac {d(x,xf_i^N)}N>0$ for
any $x\in \X$.
\end{proof}

\begin{lemma}\label{scaling splittings}
Let $S\in\overline{\mathcal T}$ be a simplicial tree and $g_i\in
Out(F_n)$ an infinite sequence such that $Sg_i\to T$ projectively. If
$\lambda_i$ is a scaling sequence such that $Sg_i/\lambda_i\to T$,
then $\lambda_i$ is bounded from below. Furthermore, if for every
$\epsilon>0$ there are conjugacy classes with length in $T$ in the
interval $(0,\epsilon)$ then $\lambda_i\to\infty$.
\end{lemma}

\begin{proof}
Let $\gamma$ be a conjugacy class with $\langle
T,\gamma\rangle=L>0$. Then $$\langle
S,g_i(\gamma)\rangle/\lambda_i=\langle Sg_i/\lambda_i,\gamma\rangle\to
L$$ and since eventually $\langle S,g_i(\gamma)\rangle\geq \eta>0$
(nonzero translation lengths in $S$ are always bounded away from 0) we
see that $\lim\inf\lambda_i\geq \frac \eta L$.
\end{proof}

\begin{lemma}\label{one simplicial two irreducibles}
If $S$ is a simplicial tree and $p,q$ are irreducible trees, $p\neq
q$, then $(S|pq)<\infty$.
\end{lemma}

\begin{proof}
To simplify notation we assume that $\{f_1,\cdots,f_k\}=\{f\}$.  We
may assume that $p,q\not\in D^+\cup D^-$.  If $(S|pq)=\infty$ then
there are infinitely many $g_i\in Out(F_n)$ such that $Sg_i^{-1}\in
D^-$ and $pg_i^{-1},qg_i^{-1}\in D^+$ (or interchange $D^-$ and
$D^+$). We may assume that sequences $Sg_i^{-1},g_i(p^*),g_i(q^*)$
converge projectively, say to $T,\Upsilon_p,\Upsilon_q$
respectively. A scaling sequence for one of $pg_i$ or $qg_i$ must go
to infinity by Proposition \ref{T1}, say for the former. Then Lemma
\ref{scaling} implies that a scaling sequence $\mu_i$ for $g_i(p^*)$
also goes to infinity ($pg_i^{-1}\in D^-$ cannot converge to $p\not\in
D^-$). Let $\lambda_i$ be a scaling sequence for $Sg_i^{-1}$. Since
$\lambda_i\mu_i\to\infty$ by Lemma \ref{scaling splittings}, it
follows that $$\langle T,\Upsilon_p\rangle=\langle \lim
Sg_i^{-1}/\lambda_i, \lim g_i(p^*)/\mu_i\rangle=\lim \langle
S,p^*\rangle/(\lambda_i\mu_i)=0$$ which is a contradiction since $T$
is close to $T^-$ and $\Upsilon_p$ is close to $(T^+)^*$.
\end{proof}

\begin{prop}
Let $S$ be a simplicial tree. Then the stabilizer $Stab(S)\subset
Out(F_n)$ acts on $\X$ with bounded orbits.
\end{prop}

\begin{proof}
Identical to the proof of Proposition \ref{orbits}.
\end{proof}

\begin{lemma}\label{splitting complex}
  If $S,S'$ are simplicial trees in $\overline{\mathcal T}$ and if
  there is a nontrivial conjugacy class $\gamma$ that is elliptic in
  both $S$ and $S'$ and that is contained in a proper free factor of
  $F_n$ then $(S|S')=0$.
\end{lemma}

\begin{proof}
Martin \cite{reiner} proved that $\eta_\gamma\in\mathcal M$ (and he
proved the converse as well).
If $(S|S')>0$ then
there is $g\in Out(F_n)$ with $Sg^{-1}\in D^-$ and $S'g^{-1}\in D^+$
(or interchange $D^+$ and $D^-$). Then $\langle
Sg^{-1},g(\gamma)\rangle =\langle S'g^{-1},g(\gamma)\rangle=0$, so 
the current $\eta_{g(\gamma)}$ has length 0 in a tree close to
$T^+$ and in a tree close to $T^-$, contradiction.
\end{proof}

\begin{prop}\label{bounded orbits}
The stabilizer of the conjugacy class of a proper free factor has
bounded orbits in $\X$.
\end{prop}

\begin{proof}
The proof is a variation of the proof of Proposition \ref{orbits}. Let
$A$ be a proper free factor of $F_n$. Fix a simplicial $F_n$-tree $S$
with $A$ fixing a vertex. Let $(p_1,p_2,p_3)\in\X$ and let $g\in
Out(F_n)$ fix $A$. We will argue that
$d((p_1,p_2,p_3),(p_1g,p_2g,p_3g))$ is bounded independently of
$g$. By Lemma \ref{one simplicial two irreducibles} we have
$N=\max_{i,j}(p_ip_j|S)<\infty$. 

Suppose there are $D$ disjoint annuli separating $p_i,p_j$ from
$p_ug,p_vg$ for some $i\neq j$ and $u\neq v$. Now consider $S$ and
$S'=Sg$. By Lemma \ref{splitting complex} no annulus separates $S$
from $S'$. Moreover, at most $N$ annuli separate $S$ from $p_i,p_j$ and at
most $N$ annuli separate $S'$ from $p_ug,p_vg$. We deduce that $D\leq
2N+2$. 
\end{proof}

\subsubsection{The complex of simplicial trees}
Lemma \ref{splitting complex} suggests the definition of another
$Out(F_n)$-complex, namely the {\it complex of simplicial trees}
$\mathcal{ST}(F_n)$. A vertex is represented by a minimal, non-free,
simplicial $F_n$-tree in $\overline{\mathcal T}$ without valence 2
vertices and all edge lengths 1. (Recall \cite{CL} that a minimal
nontrivial simplicial
$F_n$-tree is in $\overline{\mathcal T}$ if and only if it is {\it
  very small}, i.e. the edge stabilizers are cyclic, and for $g\neq 1$
we have that $Fix(g)$ does not contain a tripod and $Fix(g^m)=Fix(g)$
for all $m\neq 0$.)
Two such trees span an edge
if there is a nontrivial conjugacy class $\gamma$ that is elliptic in
both trees and such that $\gamma$ is contained in a proper free
factor.

When $n=2$ this graph is quasi-isometric to the Farey graph.

Now define $\Phi:\mathcal {ST}(F_n)\to \X$ by the rule that $\Phi(T)$ is a
triple that belongs to an orbit of uniformly bounded size (see the
discussion in Section \ref{comparison}).

\begin{prop}
$\Phi$ is coarsely well defined and it is Lipschitz.
\end{prop}

\begin{proof}
Identical to the proof of Proposition \ref{Phi}.
\end{proof}

\begin{question}
Is $\Phi:\mathcal {ST}(F_n)\to\X$ coarsely onto?
\end{question}

\begin{question}
How dependent is $\X$ on the choice of $D^{\pm}$? For example, when
$D^{\pm}$ keep getting smaller, we expect that the natural maps
between $\X$'s do not eventually become quasi-isometries. Does Remark
\ref{ken} hold in the $Out(F_n)$ world?
\end{question}

\begin{question}
Is $\X$ quasi-isometric to a tree (provided $D^{\pm}$ are sufficiently
small)?
\end{question}

We finish this section by comparing $\X$ to two other
$Out(F_n)$-complexes. 

\subsubsection{The complex of free factors}
Let $\mathcal F(F_n)$ denote {\it complex of free factors}: its
vertices are conjugacy classes of proper free factors, and its
simplices are conjugacy classes of chains (ordered by inclusion) of
proper free factors. This complex has been introduced and studied by
Hatcher and Vogtmann \cite{HV}. It is a discrete set when $n=2$ and it
is connected when the rank $n>2$.

There is a map $\Psi:\mathcal F(F_n)\to \mathcal{ST}(F_n)$ given by the rule
that $\Psi(F)$ is the Bass-Serre tree associated with a splitting
$F_n=F*F'$. This map is coarsely well-defined and Lipschitz.

\subsubsection{The splitting complex}
A tree $S\in\overline{\mathcal T}$ is a {\it splitting tree} if it is
the Bass-Serre tree of a nontrivial splitting $F_n=A*B$. The {\it
  splitting complex} is the simplicial complex $\mathcal S(F_n)$ whose
vertices are splitting trees, and a collection $S_i$ of such trees spans a
simplex if there is a simplicial $F_n$-tree $S$ with trivial edge
stabilizers such that each $S_i$ can be obtained from $S$ by
equivariantly collapsing collections of edges. When $n=2$ this complex
is a discrete set and when $n>2$ it is connected.

There is a map $\Sigma:\mathcal S(F_n)\to \mathcal F(F_n)$ that to a
splitting $A*B$ assigns $A$. When $n>2$ this map is coarsely well
defined and Lipschitz.

To summarize, for $n>2$ we have maps
$$\mathcal S(F_n)\overset{\Sigma}\to \mathcal F(F_n)\overset\Psi\to 
\mathcal{ST}(F_n)\overset\Phi\to \mathcal X$$ The map $\Sigma$ is
coarsely onto by construction. The composition $\Psi\Sigma:\mathcal
S(F_n)\to\mathcal{ST}(F_n)$ (and hence also $\Psi$) is coarsely onto,
because a finite graph of groups with cyclic edge groups representing
$F_n$ can be converted to a finite graph of groups with trivial edge
groups by a (bounded) sequence of elementary moves, see
\cite{shenitzer}\cite{swarup}. 

\begin{remark}
  If one takes $\Phi(T)$ to be the subset of $\X$ consisting of points
  whose orbit under $Stab(T)$ has diameter bounded by $K_d$ (see
  Section \ref{comparison}), then $\Phi$ becomes an equivariant coarse
  map. It follows immediately that translation lengths in $\mathcal
  {ST}(F_n)$ of fully irreducible automorphisms are positive. The same
  statement holds for $\mathcal S(F_n)$ and $\mathcal F(F_n)$. This
  fact was proved for nongeometric automorphisms in
  \cite{ilya-lustig}.
\end{remark}

\begin{question}
Are $\Sigma$ and $\Psi$ quasi-isometries? We expect that $\Sigma$ is
not. More precisely, take a pseudo-Anosov homeomorphism $f$ on a
surface with two boundary components and view it as an element of
$Out(F_n)$ (this is possible when $n>2$). Then $f$ acts with bounded
orbits on $\mathcal F(F_n)$, but we expect that $f$ has positive
translation length in $\mathcal S(F_n)$.
\end{question}

\subsection{WPD}\label{WPD}

\begin{prop}\label{p:WPD}
  The elements $f_1,\cdots,f_k$ chosen at the start of the
  construction satisfy WPD: For every $i=1,2,\cdots,k$, every $x\in \X$,
  and every $C>0$ there is $N>0$ such that
$$\{g\in Out(F_n)| d(x,xg)\leq C, d(xf_i^N,xf_i^Ng)\leq C\}$$
is finite.
\end{prop}

The proof requires two lemmas.

\begin{lemma}\label{l:scaling}
Let $U$ and $U'$ be disjoint closed sets in $\overline{\mathcal{PT}}$
and let $P\not\in U\cup U'$ be irreducible. Suppose that for $s=1,2$:
\begin{itemize}
\item there are $f_s,
f_s'\in Out(F_n)$ such that $S_s=Pf_s\in U$ and $S'_s=Pf'_s\in
U'$, 
\item
$S_1\not=S_2$ and $S_1'\not= S_2'$,
\item
there are infinitely many distinct $g_i\in Out(F_n)$ such that
$S_sg_i\in U$ and $S_s'g_i\in U'$ for all $i$, and
\item
the sequences $S_sg_i$, $S_s'g_{i}$, $g_i^{-1}S_s^*$, and
$g_i^{-1}{S_s'}^*$ converge projectively.
\end{itemize}
Then
\begin{enumerate}
\item
three of four given scaling sequences for $S_sg_i$ and $S'_sg_i$, $s=1,2$,
converge to infinity, and
\item
three of four given scaling sequences for $g_i^{-1}S_s^*$ and $g_i^{-1}{S_s'}^*$, $s=1,2$, converge to
infinity.
\end{enumerate}
\end{lemma}

\begin{proof}
Item~1 follows from Proposition~\ref{T1}. 

By taking a subsequence, we may assume that $Pg_i^{-1}$ converges
projectively. Using Proposition~\ref{T1} again, we may assume that the
scaling sequence $\mu_i$ for $Pg_i^{-1}$ converges to
infinity. (Otherwise replace $P$ by an element in the same orbit that
is also not in $U\cup U'$.)

For Item~2, we claim that if $Pg_i^{-1}\not\to S_1$ projectively then
$\mu_i$ is a scaling sequence for $g_i^{-1}S_{1}^*$. Note that $\mu_i$
is also a scaling sequence for the translated sequence
$Pg_i^{-1}f_1^{-1}$. By Lemma~\ref{scaling} $\mu_i$ is a scaling
sequence for $g_i^{-1}S_1^*=g_i^{-1}f_1^{-1}P^*$ if two conditions
hold: projectively $Pg_i^{-1}f_1^{-1}\not\to P$ and $Pf_1g_i\not\to
P$. By hypothesis, $Pf_1g_i=S_1g_i\not\to P$ and so the second
condition holds. If the first condition fails, then $Pg_i^{-1}\to
Pf_1=S_1$.

Item~2 now follows because the sequence $Pg^{-1}_i$ can converge to at
most one of $S_1$, $S_2$, $S_1'$, $S_2'$.
\end{proof}

Let $\X'$ be the Bowditch space obtained by using the $Out(F_n)$-orbit
of a single irreducible tree $P$ instead of all irreducible trees.

\begin{lemma}\label{l:qi}
$\X'$ and $\X$ are quasi-isometric.
\end{lemma}

\begin{proof}
Since the metric on $\X'$ is the restriction of the metric on $\X$, it
is enough to show that $\X'\subset \X$ is co-bounded. Let
$t=(T_1,T_2,T_3)\in\X$. Using north-south dynamics, since
$(T_i|T_lT_m)>0$ (in fact infinite) if $l\not=i\not=m$, there is
$p=(P_1,P_2,P_3)\in\X'$ such that $(P_iT_i|P_lT_m)>0$ (in fact
arbitrarily large) if $l\not= i\not= m$. Recalling that $k=0$,
Property (C1) in Section~\ref{s:crossratio} (see also the comment on
(C1) in that section) implies that $(P_iP_l|T_iT_m)=0$ for
$l\not=i\not= m$, i.e.\ $\rho(p,t)=0$.
\end{proof}

\begin{proof}[Proof of Proposition~\ref{p:WPD}]
By Lemma~\ref{l:qi}, it is enough to show that the proposition holds
for $\X'$. Denote $f=f_i$. By definition of distance in $\X'$, the condition
$d(x,y)\leq C$ is equivalent to $\rho(x,y)=max(X|Y)\leq C'$ for a suitable
$C'>0$, where $X,Y$ range over 2-element subsets of $x,y$
respectively.  Construct closed neighborhoods $U^\pm_j$ of $T_f^\pm$, $j=0,1$
and a closed neighborhood $\Omega^-$ of $\Upsilon_f^-$ so that:
\begin{itemize}
\item $U_0^\pm\supset U_1^\pm$.
\item $\langle T,\Upsilon\rangle >0$ if $T\in {U_0^-}$ and
  $\Upsilon\in \Omega^-$; moreover, a current dual to an irreducible tree
  $T'\in U_0^+$ belongs to $\Omega^-$.
\item $(U^\pm_1|M-U_{0}^\pm)>C'$.
\end{itemize}
This is possible by Proposition \ref{north-south} and the facts listed
in Section \ref{currents}.  The statement is invariant under replacing
$x$ by $xf^m$ for any $m$, so we may assume that if $x=(S_1,S_2,S_3)$
then, for two values of $s$ say $s=1,2$, $S_s\in U_1^-$. In this case,
we may also assume that $S'_s=S_sf^N\in U_1^+$ for $s=2,3$ and some
$N>0$.

Suppose that $g_1,g_2,\cdots$ is an infinite collection in $Out(F_n)$
so that $\rho(x,xg_i)\leq C'$ and $\rho(xf^N,xf^Ng_i)\leq C'$ for all
$i$. It follows
that, for each $i$, $S_1g_i$ and $S_2g_i$ belong to
$U^-_0$. Similarly, for each $i$, $S_2'g_i$ and $S_3'g_i$ belong to
$U^+_0$. Passing to a subsequence, $S_sg_i$ and
$g_i^{-1}S_s^*$, for $s=1,2$, and $S'_{s'}g_i$ and
$g_i^{-1}S_{s'}^{'*}$, for $s'=2,3$, all converge projectively.

By Lemma~\ref{l:scaling}, for either $s=1$ or 2 a scaling sequence for
$S_sg_i$ goes to infinity and for either $s'=2$ or 3 a scaling
sequence for $g_i^{-1}S_{s'}^{'*}$ goes to infinity. For convenience
assume these values are $s=1$ and $s'=3$. The other cases differ only in notation.

Let $S_1g_i/\lambda_i\to T$ and $g_i^{-1}{S'_3}^{*}/\mu_i\to
\Upsilon$. We have
$$\langle
T,\Upsilon\rangle=\langle \lim S_1g_i/\lambda_i,\lim g_i^{-1}{S'_3}^*/\mu_i\rangle = \lim
\langle S_1,{S'_3}^*\rangle/(\lambda_i\mu_i)=0$$
But $g_i^{-1}({S'_3}^*)=(S'_3g_i)^*\in \Omega^-$, so $\Upsilon\in\Omega^-$
and we have a contradiction to the second bullet in the proof.
\end{proof}

\subsection{Application to quasi-homomorphisms}

Recall that a {\it quasi-homomorphism} on a group $\G$ is a function
$\phi:\G\to\R$ such that
$\Delta(\phi):=\sup_{\gamma,\gamma'\in\G}|\phi(\gamma\gamma')-
\phi(\gamma)-\phi(\gamma')|<\infty$. The collection of all
quasi-homomorphisms on $\G$ is a vector space $QH(\G)$ that contains
bounded functions as well as homomorphisms $\G\to\R$. We denote by
$\widetilde{QH}(\G)$ the quotient of $QH(\G)$ by the subspace spanned
by bounded functions and homomorphisms. Then $\widetilde{QH}(\G)$ can
be identified with the kernel of the natural homomorphism
$H^2_b(\G;\R)\to H^2(\G;\R)$ from the second bounded cohomology of
$\G$ to the standard cohomology.

The following result was announced by Hamenst\"adt, who uses methods
of \cite{ursula}. We say that two fully irreducible automorphisms
$f,g\in Out(F_n)$ are {\it independent} if $\{T_f^+,T_f^-\}\cap
\{T_g^+,T_g^-\}=\emptyset$.

\begin{cor}\label{holy grail}
$\dim\widetilde{QH}(Out(F_n))=\infty$. Moreover, if $\Gamma<Out(F_n)$
is any subgroup that contains two independent fully irreducible
automorphisms, then $\dim\widetilde{QH}(\Gamma)=\infty$.
\end{cor}

\begin{proof}
  This follows from the arguments of \cite{BF}. In that paper the main
  theorem was proved under the assumption that every hyperbolic
  element satisfies WPD (and then the action is said to satisfy
  WPD). In fact, it suffices to know that one of the hyperbolic
  elements in the Schottky subgroup (the one generated by high powers
  of independent fully irreducible elements) satisfies WPD. Then
  \cite[Proposition 6(5)]{BF} shows that there exist hyperbolic
  elements $g_1,g_2$ with $g_1\not\sim g_2$ and then \cite[Theorem
  1]{BF} implies the result.
\end{proof}

\begin{cor}
Let $\G$ be an irreducible lattice in a semisimple Lie group of rank
$\geq 2$. If $\G\to Out(F_n)$ is an embedding, the image does not
contain any fully irreducible automorphisms.
\end{cor}

\begin{proof}
Burger and Monod proved that $\widetilde{QH}(\Gamma)=0$ \cite{BM}. By
Corollary \ref{holy grail} the image $H\subset Out(F_n)$ does not
contain two independent fully irreducible automorphisms. Now suppose
that $f\in H$ is fully irreducible. If $H$ leaves $T_f^\pm$ invariant
then $H$ and $\G$ are virtually cyclic (see \cite{gafa}), which is
impossible. If $h\in H$ does not preserve $T_f^\pm$ then $f$ and
$hfh^{-1}$ are independent fully irreducible automorphisms in $H$,
contradiction. 
\end{proof}

\begin{cor}
The Cayley graph of $Out(F_n)$ with respect to a finite generating set
contains arbitrarily large balls
consisting entirely of fully irreducible automorphisms.
\end{cor}

\begin{proof}
Fix a quasi-homomorphism $\phi:Out(F_n)\to\R$ which is unbounded and
which arises from our construction. Then there is a constant $C>0$
such that whenever $g\in Out(F_n)$ is not fully irreducible, then
$|\phi(g)|<C$. This is because $g$ has bounded orbits on $\X$ by
Proposition \ref{bounded orbits} and hence there is a uniformly
bounded orbit. On such elements $\phi$ is uniformly bounded by
construction of \cite{BF}.

Now fix $R>0$ and let $C'=\max_{x\in B(1,R)}|\phi(x)|$. Choose some
$f\in Out(F_n)$ such that $\phi(f)>C+C'+\Delta(\phi)$. Then
$\phi(g)>C$ for every $g\in B(f,R)$ so this ball consists of fully
irreducible automorphisms.
\end{proof}

\begin{remark} A similar argument shows that for every $R$ there is
  $R'$ so that every $R'$-ball contains an $R$-ball that consists
  entirely of fully irreducible automorphisms.
\end{remark}

\subsection{Dictionary}
The table below provides a correspondence between some objects
associated with $Out(F_n)$ and others associated with $MCG$.

\renewcommand\arraystretch{1.6}
\begin{center}
  \begin{tabular}{| c | c | }
    \hline
    \multicolumn{2}{|c|}{Dictionary} \\ \hline\hline
    $F_n$ & $S$ a compact surface\\ \hline
    $Out(F_n)$ & $MCG$ \\ \hline 
primitive element & non-$\partial$-parallel scc \\ \hline
free factor & connected subsurface \\ \hline
fully irreducible $f$ & pseudo-Anosov $f$ \\ \hline
simplicial tree & multicurve \\ \hline
    $\overline{\mathcal{T}}$ &  $\mathcal{ML}$\\ \hline
    $\mathcal{MC}(F_n)$& $\mathcal{MC}$ \\ \hline
    $T_f^+$ & $\Lambda_f^+$\\ \hline
    $\mathcal{M}(F_n)=\overline{\{\eta_\gamma\mid\gamma\mbox{ primitive}\}}$ & $\overline{\{\eta_\gamma\mid\gamma\mbox{ a non-$\partial$-parallel scc}\}}$\\ \hline
    $\overline{\mathcal{T}}\times\mathcal{MC}(F_n)\overset{\langle\cdot ,\cdot\rangle}{\to} [0,\infty)$ & $\mathcal{ML}\times\mathcal{MC}\overset{\langle\cdot ,\cdot\rangle}{\to} [0,\infty)$ \\ \hline
 \end{tabular}
\end{center}

\begin{remark}
  The space $\mathcal{MC}$ of measured currents on a surface and an
  intersection pairing
  $\langle\cdot,\cdot\rangle:\mathcal{MC}\times\mathcal{MC}\to
  [0,\infty)$ was introduced by Bonahon \cite{bonahon1}. He also
  produces an embedding $\mathcal{ML}\to\mathcal{MC}$ (whose image in
  $\mathcal {PMC}$ is the closure of the set of $\eta_\gamma$'s) and
  the pairing in the table is obtained by restriction. A {\it
    multicurve} is a measured lamination with support a collection of
  disjoint simple closed curves (scc's). According to Skora
  \cite{skora}, $\mathcal{PML}$ can be identified with projectivized
  space of small $\pi_1(S)$-trees in which boundary curves are
  elliptic. The subspaces of simplicial trees in $\mathcal{PT}$ and
  $\mathcal{PML}$ are dense.
\end{remark}

A version of Corollary~\ref{T2}
holds for surfaces. Fix a complete hyperbolic structure on the
interior of $S$ and by $|\alpha|$ denote the hyperbolic length of the
closed geodesic homotopic to $\alpha$.

\begin{thm}\label{T2surface}
Let $f$ and $g$ be pseudo-Anosov and assume
$\Lambda^+_f\not=\Lambda^+_g$. There is $\delta>0$ such that for all
non-boundary-parallel simple closed curves $\alpha$ we have either
$\langle \Lambda^+_f,\alpha\rangle\ge\delta|\alpha|$ or $\langle
\Lambda^+_g,\alpha\rangle\ge\delta|\alpha|$.
\end{thm}

The ingredients of the proof are that simple closed geodesics never
enter a neighborhood of any cusp and that on the complement of these
neighborhoods the hyperbolic metric is comparable to the Euclidean
metric with cone singularities determined by $\Lambda^+_f$ and
$\Lambda^+_g$.

Given Theorem~\ref{T2surface}, a proof of
Proposition~\ref{wpd.surface} and an alternate proof of
Theorem~\ref{MCG} can be obtained by using the dictionary to
translate the proofs in the $Out(F_n)$ case.

\bibliography{./ref}

\end{document}